\documentclass[a4paper,12pt]{article}
\usepackage{amsmath, amsthm}
\usepackage{mathtools}
\usepackage{amssymb}
\usepackage{amscd}
\usepackage{epic, eepic}
\usepackage{url}
\usepackage{color}
\usepackage[utf8]{inputenc} 
\usepackage{comment}
\usepackage{todonotes}
\usepackage{graphicx}
\usepackage{epstopdf}
\usepackage{enumerate}
\usepackage{array}
\usepackage{tabularx}
\usepackage{tikz}
\usepackage{enumitem}
\usepackage{tcolorbox}
\usepackage{pdflscape}
\usepackage{svg}
\usepackage{hyperref}

\usepackage{caption}
\usepackage{subcaption}
\usepackage{makecell}

\newcommand{\Mod}[1]{\ (\mathrm{mod}\ #1)}

\newtheorem{theorem}{\bf Theorem}[section]
\newtheorem{lemma}[theorem]{\bf Lemma}
\newtheorem{cor}[theorem]{\bf Corollary}

\newtheorem{problem}[theorem]{\bf Problem}

\newtheorem{nota}[theorem]{\bf Notation}

\newtheorem{remark}[theorem]{\bf Remark}
\newtheorem{defi}[theorem]{\bf Definition}

\newtheorem{conjecture}[theorem]{\bf Conjecture}

\title{Extending edge colorings of distance-3 matchings in the Cartesian product of graphs}

\date{}

\author{Pál Bärnkopf \thanks{Alfréd Rényi Institute of Mathematics, Budapest, Hungary. Partially supported by the Counting in Sparse Graphs Lendület Research Group. This project was partially supported by the ERC grant ERMiD.
E-mail: {\tt barpal@student.elte.hu}} 
\and 
Ervin Győri \thanks{Alfréd Rényi Institute of Mathematics, Budapest, Hungary, Partially supported by the National Research, Development and Innovation Office NKFIH, grants K132696 and SNN135643, 	E-mail: {\tt gyori@renyi.hu}}
}
 
\begin{document}

\maketitle

\begin{abstract}
We investigate the problem of extending partial edge colorings in Cartesian products of graphs, with a particular focus on cases where the precolored edges form a matching. Casselgren, Granholm, and Petros conjectured that any precolored distance-3 matching in $G = C^d_{2k}$ can be extended to a $2d$-edge coloring. In this paper, we prove a theorem that implies this conjecture. Especially, our main result establishes that a precolored distance-3 matching in the Cartesian product of certain class 1 graphs can be extended to an edge coloring that uses at most as many colors as the chromatic index, provided that certain degree conditions are satisfied. In the second part of the paper, we extend these results to Cartesian products of other types of graphs as well.

\textit{Keywords: Precoloring extension; Edge coloring; Bipartite graph; Cartesian product} 
\end{abstract}

\section{Introduction}

In this paper, we deal with proper edge colorings of graphs. Throughout the paper, instead of proper edge coloring, we often say just coloring. The edge chromatic number (or chromatic index) of a graph $G$ is denoted by $\chi'(G)$. According to Vizing's Theorem \cite{Vizing}, if $G$ is a simple graph with maximum degree $\Delta(G)$, then its chromatic index is either $\Delta(G)$ or $\Delta(G)+1$. Graphs with $\chi'(G)=\Delta(G)$ are called \textit{Class 1 graphs}. (In particular, bipartite graphs are proved to be Class 1 graphs.) An (edge) \textit{precoloring} (or partial edge coloring) of a graph $G$ is a proper edge coloring of some edge set $E_0 \subseteq E(G)$. 

We call a precoloring of some edge set $E_0$ \textit{extendable} if there is a proper edge coloring with $\chi'(G)$ colors such that the color of the edges in $E_0$ are the prescribed colors. Such a coloring is called an \textit{extension} of the precoloring.

Let $G$ be a graph, and let us prescribe some color for some edges of $G$. We call a (proper edge) coloring of $G$ with $\chi'(G)$ colors \textit{admissible} if it is an extension of the precoloring.

We recall that the problem of extending a given edge precoloring is an NP-hard problem, even for 3-regular bipartite graphs. \cite{Easton}, \cite{Fiala}

\begin{defi}
    \textup{The \textit{Cartesian product $G_1 \square G_2$ of graphs} $G_1=(V_1,E_1)$ and $G_2=(V_2,E_2)$ is a graph whose vertex set is the Cartesian product $V(G_1) \times V(G_2)$ and two vertices $(u_1,u_2)$ and $(v_1,v_2)$ are adjacent in $G_1 \square G_2$ if and only if either}
    \begin{itemize}
        \item \textup{$u_1=v_1$ and $u_2$ is adjacent to $v_2$ in $G_2$ or}
        \item \textup{$u_2=v_2$ and $u_1$ is adjacent to $v_1$ in $G_1$.}
    \end{itemize}
\end{defi}

\begin{nota}
    \textup{$G^d$ denotes the $d$-th power of the Cartesian product of $G$ with itself.}
\end{nota}

\begin{problem} \label{problem}
    Let $G$ and $H$ be bipartite (or Class 1) graphs with maximum degrees $\Delta(G)$ and $\Delta(H)$, respectively. We would like to color the edges of the Cartesian product $G \square H$ with $\Delta(G) + \Delta(H)$ colors, even if we prescribe the color of some edges. Under what conditions does the graph have an admissible coloring?
\end{problem}

A natural condition is that the precolored edges are far from each other. The problem of extending a partial edge coloring with distance constraints on the precolored edges goes back to the work of Albertson and Moore \cite{Albertson}.

\begin{defi}
    \textup{In a connected graph, the \textit{distance $d(x,y)$ between vertices} $x$ and $y$ is the length of the shortest $(x,y)$-path and the \textit{distance $d(e,f)$ of edges} $e=xy$ and $f=zw$ is $\min \{d(x,z),d(x,w),d(y,z),d(y,w) \}$.}
\end{defi}

\begin{defi}
    \textup{The matching $M$ is a \textit{distance-k matching}, if for every $m_1, m_2 \in M: d(m_1,m_2) \geq k$.}
\end{defi}

Casselgren, Granholm, and Petros \cite{Casselgren} proved the following.

\begin{theorem} \label{distance-4}
    If $E_0 \subseteq E(G)$ is a distance-4 matching of $G = C^d_{2k}$, then any $2d$-coloring of $E_0$ can be extended to a proper $2d$-edge coloring of $G$.
\end{theorem}

Some related theorems are proved in \cite{cassel2}, \cite{cassel3}, \cite{cassel4}, 
\cite{cassel5}, \cite{Edwards} and \cite{Marcotte}.

In the paper \cite{Casselgren}, the authors raised a conjecture as follows.

\begin{conjecture} \label{conjecture}
    If $E_0 \subseteq E(G)$ is a distance-3 matching of $G = C^d_{2k}$, then any $2d$-coloring of $E_0$ can be extended to a proper $2d$-edge coloring of $G$.
\end{conjecture}

\begin{remark} \label{dist-2}
    \textup{It is easy to give  an example of a graph in which the minimal distance of the precolored edges is at least two, but does not have an admissible coloring  (e.g., see \cite{Casselgren}).}
\end{remark}

\begin{remark}
    \textup{It is clear that Theorem \ref{distance-4} also holds for the Cartesian product $G \square H$ of arbitrary Class 1 graphs $G$ and $H$, instead of $C^d_{2k}$. We may assume that $G$ and $H$ are regular graphs (see Lemma \ref{regularization}), and under this assumption, the proof of Theorem \ref{distance-4} applies directly.}
\end{remark}

We recall that the Cartesian product of a Class 1 graph with any other graph is also a Class 1 graph. Therefore, in the case of a product with multiple factors, if at least one of the factors is Class 1, then the entire product is Class 1 as well.

The main theorem of this paper is the following theorem that implies Conjecture \ref{conjecture}.

\begin{theorem} \label{general}
    Let $k$ be an integer with $k \geq 2$, and let $G_i$ be Class 1 graphs for each $i \in \{1, \dots, k\}$, with maximum degree $\Delta(G_i)$, respectively. Suppose that some edges of the Cartesian product $G_1 \square G_2 \square \cdots \square G_k$ are precolored using at most $\sum_{i=1}^k \Delta(G_i)$ colors, ensuring that the distance between any two precolored edges is at least three. If $\sum_{i=1}^k \Delta(G_i)$ is even and satisfies $2\Delta(G_i) \le \sum_{j=1}^k \Delta(G_j) \quad \forall i \in \{1, \dots, k\}$, then this precoloring is extendable.
\end{theorem}

\begin{remark}
    \textup{Theorem \ref{general} solves Problem \ref{problem} in the case where $\Delta(G)=\Delta(H)$ and also Conjecture \ref{conjecture}, as the degree conditions are clearly satisfied in both cases.}
\end{remark}

In the second section, we also study the problem in the case of the Cartesian product of a bipartite graph and an edge. Galvin's theorem \cite{Galvin} is used in the proof of this statement, so the proof does not work for the Cartesian product of a Class 1 graph and an edge. However, if the List Coloring Conjecture holds for Class 1 graphs, then the proof automatically extends to this case as well. The final section presents some related results on Cartesian products with odd cycles.

\section{Admissible edge-colorings of the Cartesian \\ product of bipartite graphs}

In this section, we prove Theorem \ref{general}. The basic idea of the proof is an improved and adapted version of the proof of Theorem 5.1 in \cite{Casselgren} on extending partial edge colorings of iterated Cartesian products of cycles and paths.

\begin{nota}
    \textup{For the edges $xy=e \in E(G)$ and $pq=f \in E(H)$, $xy \square pq = e \square f$ denotes the subgraph (a 4-cycle) induced in $G \square H$ by $\{(x,p),(x,q),(y,p),(y,q)\}$.}
\end{nota}

One of the basic concepts in the paper is as follows.

\begin{defi} (The brick-neighborhoods of an edge in the Cartesian product)

    \textup{Let $yz \in E(G)$ be an edge. Then take a path $xyzw$ in $G$ and an edge $pq \in E(H)$. Then a \textit{brick-neighborhood of the edge $(y,q)(z,q)$} is the Cartesian product of the path $xyzw$ and the edge $pq$ (denoted $xyzw \square pq$). Symmetrically, we can define brick-neighborhoods starting with an edge in $H$ (i.e., we take a 3-path in $H$ and an edge in $G$).}
\end{defi}

 \begin{defi}
    \textup{The edge $(y,p)(z,p)$ and the edges in the brick-neighborhood incident to $(y,q)$ or $(z,q)$ (altogether six edges) are called \textit{internal edges}, the other (four) edges of the brick-neighborhood are called \textit{external edges}.}
\end{defi}

\begin{remark}
    \textup{Notice that an edge $e \in E(G \square H)$ typically has several brick-neigh\-bor\-hoods depending on the choice of the edges $xy, zw$ and $pq$.}
\end{remark}

\begin{remark} \label{internal-external}
    \textup{Notice that if the distance between any two precolored edges is at least three, then the edge set of two brick-neighborhoods of different precolored edges are either edge-disjoint or share an external edge.}
\end{remark}

The following definition will also be useful in the proofs.

\begin{defi}(Canonical coloring of the Cartesian product of graphs)

    \textup{Let us fix a proper coloring of each Class 1 graph $G_i$ for $i \in \{1, \dots, k\}$ using $\Delta(G_i)$ colors, respectively, resulting in a total of $\sum_{i=1}^k \Delta(G_i)$ distinct colors. We define the canonical coloring of the graph $G = G_1 \square G_2 \square \cdots \square G_k$ as follows. If an edge $(x_1,\dots,x_k)(y_1,\dots,y_k) \in E(G)$ connects two vertices that differ in the $j$th coordinate, we assign it the same color as the edge $x_jy_j$ in $G_j$.} 
\end{defi}

The canonical coloring is a proper coloring with $\sum_{i=1}^k \Delta(G_i)$ colors, and we apply certain local operations to make it admissible.

\begin{lemma} \label{degrees}
    Given $k \geq 2$ disjoint sets $H_1, \dots, H_k$ and their union $H = H_1 \cup \dots \cup H_k$ where $|H|$ is even and satisfies $2|H_i| \le |H|$ for all $i \in \{1, \dots, k\}$. The elements of $H$ can be partitioned into $|H|/2$ disjoint pairs such that no pair is a subset of any $H_i$.
\end{lemma}

\begin{proof}
    The proof proceeds by induction on $|H|$. For $|H| < 8$, the statement holds trivially. Now, assume that the statement holds for all sets $H'$ of size smaller than $8$, and consider a set $H$ that satisfies the given conditions. Select one element from each of the two largest sets and pair them together. After this pairing, the remaining sets still satisfy the conditions: The size of $|H|$ decreases by $2$, so it remains even. The sizes of the two largest sets decrease by $1$ each, which means that their doubled sizes decrease by $2$. The third largest set could not have had more than $\frac{|H|-2}{2}$ elements before pairing, because otherwise the total size of $H$ would be at least $3\frac{|H|-2}{2}$, which exceeds $|H|$, contradicting the assumption that $|H| \geq 8$. Thus, after removing the paired elements, the remaining elements still satisfy the conditions of the statement and we can complete the pairing.
\end{proof}

\begin{lemma} \label{regularization}
    Let $k$ be an integer with $k \geq 2$, and let each $G_i$ (for $i \in \{1, \dots, k\}$) be a Class 1 graph. By adding extra vertices and edges to each graph $G_i$, we can make them $\Delta(G_i)$-regular while preserving their Class 1 property and ensuring that if the distance between any two edges in the Cartesian product of the original graphs $G_i$ is at least 3, then the same holds in the Cartesian product of the regularized version as well.
\end{lemma}

\begin{proof}
    It suffices to show that a single factor can be made regular. This procedure can then be applied individually to each factor, thereby proving the statement.

    We describe a possible procedure for adding these extra vertices and edges to a graph $H$. Starting with $H$, whenever a vertex has a degree less than $\Delta(H)$, we duplicate the entire graph and connect each such vertex to its counterpart in the copy. This ensures that any vertex of degree less than $\Delta(H)$ gains exactly one new neighbor, increasing its degree without exceeding $\Delta(H)$. This operation increases the minimum degree without raising the chromatic index. We repeat this process until the graph becomes regular. (The minimum degree increases in every step and the maximum degree is unchanged, so the process stops in $\Delta(H)-\delta(H)$ steps.)

    We claim that the graph remains Class 1 at every step. $H$ has a $\Delta(H)$-edge-coloring and we can color both copies identically. Furthermore, new edges are only added between vertices whose degrees are still below $\Delta(H)$, so at each such vertex there is an unused color. We can assign these unused colors to these new edges without increasing the total number of colors used.

    Now, let $G$ and $H$ be Class 1 graphs and let $H^*$ denote the regular extension of $H$ obtained by this process. We show that if the distance between two edges $e_1, e_2$ in $G \square H$ is at least 3, it remains at least 3 in $G \square H^*$. Assume, for contradiction, that $d_{G \square H}(e_1, e_2) \geq 3$ but $d_{G \square H^*}(e_1, e_2) \leq 2$.

    Since no new edges were added between original vertices, the distance cannot be 1. (And the distance cannot be 0, of course.) If the distance is 2, there must exist a vertex $(a, b)$ in $G \square H^*$ and vertices $(u, v) \in e_1$ and $(x, y) \in e_2$ such that $(u, v) \sim (a, b)$ and $(a, b) \sim (x, y)$. (We  use $\sim$ to indicate adjacency.)
    
    \begin{enumerate}
        \item \textbf{Case 1: $(a, b)$ is an original vertex in $G \square H$.} Since the construction only adds edges between original and new vertices (or between two new vertices), the adjacency relations between original vertices are unchanged. This implies $d_{G \square H}((u, v), (x, y)) = 2$, contradicting $d_{G \square H}(e_1, e_2) \geq 3$.
        
        \item \textbf{Case 2: $(a, b)$ is a new vertex.} By the definition of the Cartesian product, for the adjacency $(u, v) \sim (a, b)$ where $(u, v)$ is original and $(a, b)$ is new, we must have $u = a$ and $v \sim  b$ in $H^*$. Similarly, for $(x, y) \sim (a, b)$, we must have $x = a$ and $y \sim  b$ in $H^*$. In our regularization procedure, whenever we duplicate a graph, we only add edges between a vertex and its direct copy (or between two new vertices). By induction, any new vertex $b \in V(H^*) \setminus V(H)$ is connected to at most one vertex in the original vertex set $V(H)$, which is its unique "preimage" (the vertex from which it was originally derived). Since $v \sim  b$ and $y \sim  b$ in $H^*$, and $v, y \in V(H)$ while $b \notin V(H)$, the construction implies that $v$ and $y$ both must be the unique preimage of $b$ in $V(H)$. It follows that $u = x$ and $v = y$. Thus, $(u, v) = (x, y)$, which means that $e_1$ and $e_2$ share a vertex, so their distance is 0, a contradiction.
    \end{enumerate}    

    Since all cases lead to a contradiction, our assumption was false: the distance cannot decrease below 3 in the regularized version.
\end{proof}


\begin{proof}[Proof of Theorem \ref{general}]
    Since the components of the graph can be colored independently, we may assume that $G = G_1 \square \cdots \square G_k$ is connected. Consider a canonical coloring of the graph $G$. This is a proper coloring with $\Delta(G)$ colors, though it is not necessarily an admissible coloring. Now, we modify the coloring step by step, preserving the number of colors and ensuring that the coloring is proper. If a precolored edge has the correct color, then we will not change it. If the prescribed color of a precolored edge is not its color in the canonical coloring, then we will correct it by some local operations. So, we finally get an admissible coloring.

    We may assume that the graphs $G_i$ are regular according to Lemma~\ref{regularization}. Notice that if this new graph has an admissible coloring, then its restriction for $G$ is an admissible coloring, too. 

    For the edges $e \in E(G_i)$ and $f \in E(G_j)$, the subgraph $e \square f$ of $G$ is called a \textit{square}. In the canonical coloring, the edges of each square are colored with two colors. A {\it rotation} in a 2-colored square is the swap of the colors of its edges. (On the other edges, the coloring is not changed.) After the rotation in a square, the number of colors does not change and the coloring remains proper. Since in each step we perform a sequence of rotations, these two constraints (the number of colors is $\Delta(G)$ and the coloring is proper) will be fulfilled throughout. So we pay attention to three things:
    \begin{enumerate}
        \item Only rotate a square with opposite edges of the same color. (The rotation is defined on 2-colored squares only!)
        \item Do not change the color of a precolored edge if it has its prescribed color.
        \item The color of every precolored edge should be the prescribed color after some steps.
    \end{enumerate}

    From now on (unless otherwise stated), let the coordinates of a vertex $v$ be $v_1, \dots, v_k$.

    \begin{nota}
        \textup{Let $v = (v_1, \dots, v_k) \in V(G)$. For $j \in \{1, \dots, k\}$ and a value $x$, we denote by $v^{(j:x)}$ the vertex $(v_1, \dots, v_{j-1}, x, v_{j+1} \dots, v_k)$. Similarly, $v^{(i:x, j:y)}$ denotes the vertex obtained from $v$ by replacing the $i$-th and $j$-th coordinates by $x$ and $y$, respectively.}
    \end{nota}

    STEP 1: Consider an edge $e = v^{(j:x)}v^{(j:y)} \in E(G)$ whose prescribed color appears in the coloring of $G_l$ for some $l \neq j$. Since $G_l$ is regular, there exists an edge $v_lz \in E(G_l)$ colored with this color. By performing a rotation on the square with vertices $\{v^{(j:x)}, v^{(j:y)}, v^{(j:x,l:z)}, v^{(j:y,l:z)}\}$, the edge $e$ receives its prescribed color. This operation modifies only the internal edges of every brick-neighborhood of $e$, ensuring that the coloring of every brick-neighborhoods of all other precolored edges remains unchanged (see Remark \ref{internal-external}).

    STEP 2: Consider an edge $e = v^{(j:x)}v^{(j:y)} \in E(G)$ whose prescribed color appears in the coloring of $G_j$, but differs from the color assigned to $e$ in the canonical coloring. Since $G_j$ is regular, there must exist edges $sx, yt \in E(G_j)$, which have this prescribed color in the canonical coloring. Now, choose an arbitrary edge $v_iw$ of the graph $G_i$ for some $i \neq j$. To achieve the desired color on $e$, we perform a sequence of rotations (in the following order) on the squares defined by these vertex sets:
    \begin{enumerate}
        \item $\{v^{(j:s)}, v^{(j:x)}, v^{(j:s,i:w)}, v^{(j:x,i:w)}\}$
        \item $\{v^{(j:y)}, v^{(j:t)}, v^{(j:y,i:w)}, v^{(j:t,i:w)}\}$
        \item $\{v^{(j:x)}, v^{(j:y)}, v^{(j:x,i:w)}, v^{(j:y,i:w)}\}$
    \end{enumerate}

    In this case, the coloring changes on the external edges as well. Therefore, we aim to select for each precolored edge of this type one of its brick-neighborhoods, so that the selected brick-neighborhoods are pairwise edge-disjoint. If we can do this, then we can achieve an admissible coloring by applying STEP 2 to the brick-neighborhoods chosen in this way.

    Let $H_i$ be the set of colors used in the coloring of $G_i$ for each $i = 1, \dots, k$, and let $H = \bigcup H_i$. The sets $H_i$ satisfy the conditions of Lemma~\ref{degrees}; therefore, there exists an involution $\varphi: H \to H$ (i.e. $\varphi(\varphi(x))=x$) such that if $x \in H_i$, then $\varphi(x) \notin H_i$. (This means that every pair of colors belongs to the colorings of two distinct graphs.)
    
    In STEP 2, the index $i$ and the vertex $w$ can be freely chosen, provided that $i \neq j$ and $v_iw \in E(G_i)$. We shall now provide a specific choice that guarantees that the brick-neighborhoods used are disjoint. For a precolored edge $e = v^{(j:x)}v^{(j:y)} \in E(G)$ with prescribed color $c$, let $sx, yt \in E(G_j)$ be the edges with color $c$ in the canonical coloring. We select $i$ and $w$ so that the color of $v_iw \in E(G_i)$ is $\varphi(c)$. Consider the brick-neighborhood $B$ defined by the vertex set $$\{v^{(j:s)}, v^{(j:x)}, v^{(j:y)}, v^{(j:t)}, v^{(j:s, i:w)}, v^{(j:x, i:w)}, v^{(j:y, i:w)}, v^{(j:t, i:w)}\}$$ 
    
    Let $f(B)=\{c;\varphi(c)\}$ denote the pair of colors associated with $B$. These are exactly the colors assigned to the edges of the squares involved in the first two rotations of STEP 2. For any two such brick-neighborhoods $B_1$ and $B_2$ (corresponding to precolored edges), sets $f(B_1)$ and $f(B_2)$ are either identical or disjoint. Specifically, if $c_2 \in \{c_1,\varphi(c_1)\}$, then $\{c_1,\varphi(c_1)\}=\{c_2,\varphi(c_2)\}$; otherwise, $\{c_1,\varphi(c_1)\} \cap \{c_2,\varphi(c_2)\} = \emptyset$.

    Finally, we show that any two selected brick-neighborhoods are edge-disjoint. Suppose, for contradiction, that two brick-neighborhoods share an external edge (by Remark \ref{internal-external}). This implies that the two brick-neighborhoods contain a pair of (distinct) squares that share exactly one edge. Such squares can only be those involved in the first two rotations of STEP 2; therefore, the colors assigned to their edges in the canonical coloring belong to $f(B_1)$ and $f(B_2)$, respectively. As shown previously, these sets are either identical or disjoint. Since the two squares share an edge, which must have a unique color $c^*$, it follows that $c^* \in f(B_1) \cap f(B_2)$, which means $f(B_1) = f(B_2)$. However, in this case, the other two pairs of edges incident to the shared edge in each square would have the same color ($\varphi(c^*)$). Because the squares are distinct and share only one edge, this configuration would create two incident edges of the same color, contradicting the properness of the coloring (see Figure~\ref{contradiction}).
\end{proof}

\begin{figure}[ht]
\centering
\begin{tikzpicture}[scale=0.8]

    \draw (0,0) rectangle (8,3);

    \draw[line width=2pt] (4,0) -- (4,3);

    \node at (2,3.4) {$\varphi(c^*)$};
    \node at (6,3.4) {$\varphi(c^*)$};

    \node at (2,-0.4) {$\varphi(c^*)$};
    \node at (6,-0.4) {$\varphi(c^*)$};

    \node at (4.5,1.6) {$c^*$};
    \node at (-0.5,1.6) {$c^*$};
    \node at (8.5,1.6) {$c^*$};

\end{tikzpicture}
\caption{Two squares from distinct brick-neighborhoods sharing an edge}
\label{contradiction}
\end{figure}

\begin{remark}
    \textup{The above proof for two graphs relies on the equality $\Delta(G)=\Delta(H)$. We believe that the statement holds in general as well. It would be a promising step to prove the theorem at least for some values of $\Delta(G),\Delta(H)$ that are not equal. A natural first step is to consider cases where $\Delta(G)$ and/or $\Delta(H)$ is small. The following theorem settles the very first case $\Delta(H)=1$ where $G$ is bipartite.}
\end{remark}

\begin{theorem} \label{bipartite-times-edge}
    Let $G$ be a bipartite graph with maximum degree $\Delta(G)$. Suppose that the color of some edges of the Cartesian product $G \square K_{2}$ is prescribed using at most $\Delta(G)+1$ colors such that the distance between any two precolored edges is at least three. Then, this precoloring is extendable.
\end{theorem}

\begin{proof}
    Let $V(K_2)=\{a,b\}$. To prove the theorem, we will consider a list edge-coloring problem for (a subgraph of) $G$. First, assign the list of colors $\{1,2,...,\Delta(G)+1\}$ to every edge of $G$. If an edge $(u,a)(v,a)$ or $(u,b)(v,b)$ has a prescribed color (note that both cannot be precolored simultaneously), then we remove the edge $uv$ from $G$ and remove the prescribed color from the lists of all edges adjacent to $uv$. If there is a prescription on an edge $(u,a)(u,b)$, then we remove the prescribed color from the list of all edges incident to $u$.
 
    The resulting graph $G'$ remains bipartite, with a maximum degree of at most $\Delta(G)$. Moreover, each color list still contains at least $\Delta(G)$ colors since the distance condition ensures that at most one color was removed from each list. By Galvin's theorem \cite{Galvin}, $G'$ admits a proper list-coloring from these lists. This coloring of $E(G')$ can then be extended by incorporating the prescribed colors to obtain a coloring of $E(G)$ in both copies of $G$ within the Cartesian product. 

    Since the set of colors assigned to the edges incident to $u$ is the same in both copies, there is always a remaining color available to color the edge $(u,a)(u,b)$. If $(u,a)(u,b)$ had a prescribed color, the remaining available color must be the prescribed one, ensuring that the final coloring is admissible.
\end{proof}

\begin{remark}
    \textup{A similar proof also works for the Cartesian product $C_{2k+1} \square K_{2}$, we can greedily color the appropriately precolored $C_{2k+1}$ with $3$ colors instead of using Galvin's theorem.}
\end{remark}

By repeated applications of Theorem \ref{bipartite-times-edge} we get the corollary as follows.

\begin{cor} \label{c_4}
    Let $G$ be a bipartite graph with maximum degree $\Delta(G)$. Suppose that the color of some edges of the Cartesian product $G \square K_{2}^{\alpha}$ (especially for $G \square C_4$) is prescribed using at most $\Delta(G)+\alpha$ colors, so that the distance between any two precolored edges is at least three. Then this precoloring is extendable.
\end{cor}

Notice that Corollary \ref{c_4} implies Conjecture \ref{conjecture} for $k=2$.

\section{Admissible edge-colorings of the Cartesian \\ product $C_{2k+1} \square C_{2l+1}$}

Conjecture \ref{conjecture} is about the product of even cycles. A natural continuation of this is to also examine the extendability of precolorings in the case of products of odd or odd and even cycles. We can settle the case of two odd cycles.

\sloppy
\begin{theorem} \label{odd-times-odd}
    Suppose that the color of some edges of the Cartesian product $C_{2k+1} \square C_{2l+1}$ is prescribed using at most $5$ colors such that the distance between any two precolored edges is at least three. Then this precoloring is extendable.
\end{theorem}

\begin{remark}
    \textup{This is the strongest possible statement. The graph $C_{2k+1} \square C_{2l+1}$ has an odd number of vertices, which means that it does not contain any perfect matching. Therefore, it cannot be 4-edge-colorable, as that would require its edges to be divided into four perfect matchings.}
\end{remark}
\fussy

\begin{proof}[Proof of Theorem \ref{odd-times-odd}]

    \begin{nota}
        \textup{Label the vertices of the graph $C_{2k+1} \square C_{2l+1}$ with the labels $\{(i,j) \mid i=0 \dots 2k; j=0 \dots 2l\}$. The vertices $(x_1,x_2)$ and $(y_1,y_2)$ are joined by an edge if and only if $x_1=y_1$ and $|x_2-y_2| \equiv 1 \Mod{2l+1}$ or $x_2=y_2$ and $|x_1-y_1| \equiv 1 \Mod{2k+1}$.}
    \end{nota}

    There exists a square that is not part of the brick-neighborhood of any precolored edge. If there exists a precolored edge $(i,j)(i,j+1)$, then the square with vertices $\{(i+1,j), (i+2,j), (i+2,j+1), (i+1,j+1) \}$ is such a square, by the distance constraint. Without loss of generality, we may assume that $C_0$, with vertices $\{ (2k,2l), (2k,0), (0,0), (0,2l)\}$ is such a square.

    Let $E_0$ denote the set of edges, defined as $E_0 = \{(i,2l)(i,0) \mid i=0 \dots 2k\} \cup \{(2k,j)(0,j) \mid j=0 \dots 2l\}$. The deletion of $E_0$ from $C_{2k+1} \square C_{2l+1}$ yields a subgraph isomorphic to the grid graph $P_{2k+1} \square P_{2l+1}$, which we denote by $L$.

    We fix a set of five available colors, $\{1, 2, 3, 4, 5\}$. We first construct a proper edge-coloring of the grid $L$ using the first four colors. We then apply the extension method described in the proof of Theorem~\ref{general} to satisfy any prescriptions from $\{1, 2, 3, 4\}$ on the edges of $L$. (The algorithm can be applied, since the paths are bipartite graphs with maximum degree 2.) Before we begin the algorithm, we start with the canonical coloring of $L$, in which all horizontal edges of $L$ (i.e., edges whose second coordinate changes) are colored alternately with colors $1$ and $2$, while all vertical edges of $L$ (i.e., edges whose first coordinate changes) are colored alternately with colors $3$ and $4$. We will use not only the fact that the precoloring has an extension on the grid $L$, but also that this extension is obtained by applying the method described in the proof of Theorem~\ref{general}. In particular, this guarantees that any edge sufficiently far from every precolored edge keeps its color from the canonical coloring, since the algorithm leaves such edges unchanged. Moreover, if an edge $e$ in $E_0$ is horizontal (i.e., its endpoints differ in the second coordinate), then the color of a horizontal neighboring edge of $e$ does not appear at the other endpoint of $e$; and if the edge $e$ in $E_0$ is vertical (i.e., its endpoints differ in the first coordinate), then the color of a vertical neighboring edge of $e$ does not appear at the other endpoint of $e$ in the canonical coloring. (We use here the assumption that the graph is a Cartesian product of odd cycles.) Since we use the algorithm described in Theorem~\ref{general}, this property is preserved after the modification of the coloring, provided that $e$ has no neighboring precolored edge prescribed one of the colors $1, 2, 3$, or $4$.

    Afterward, the edges of $E_0$ are colored with the fifth color. (This will not be a proper coloring, but apart from the vertices of the square $C_0$, it is proper.) The coloring of the cycle $C_0$ will be adjusted at the end of the proof. We first satisfy the precoloring constraints on the grid $L$ and on the edge set $E_0 \setminus E(C_0)$ (which contains all precolored edges of $E_0$ by the choice of $C_0$).

    We consider the precolored edges of $L$ in three cases: those with no neighboring edge in $E_0$, those with exactly one neighboring edge in $E_0$, and those with exactly two neighboring edges in $E_0$. No precolored edge is adjacent to $C_0$, and since $E_0 \setminus E(C_0)$ is a matching, these are the only possibilities.

    \begin{enumerate}
        \item If an edge in $L$ is precolored with $5$ and has no neighboring edge in $E_0$, we color it with $5$ and no further changes are needed.
        \item If an edge in $L$ is precolored with $5$ and has exactly one neighboring edge in $E_0$, we swap the colors of these two edges. Since the only edge colored $5$ adjacent to the precolored edge is the edge involved in the swap, assigning color $5$ to the precolored edge does not create conflict. The other swapped color cannot cause any problem either, as, according to an earlier observation, it cannot appear at the other endpoint of the edge in $E_0$.
        \item If an edge in $L$ is precolored with $5$ and has two neighboring edges in $E_0$, then the endpoints of these three edges form a square. This square is colored with exactly two colors (the opposite edges of the square receive the same color, by the distance condition and because we use the algorithm described in the proof of Theorem~\ref{general}). We then swap the two colors appearing on this square (or, using the terminology of the proof of Theorem~\ref{general}, perform a rotation).
    \end{enumerate}

    We fix the coloring of the precolored edges in $E_0$ in two steps according to the following criterion: if an edge is precolored with $c \neq 5$, then it has either one or two neighboring edges colored $c$ (indeed, the coloring is proper on the entire graph except at the vertices of $C_0$, and no precolored edge is incident with any of those vertices).

    \begin{enumerate}
        \item If an edge $e \in E_0$ is precolored with $c \neq 5$ and has exactly one neighboring edge colored $c$, we swap the colors of these two edges. If $e$ is assigned color $c$, this cannot cause any problem, since $e$ has only one neighboring edge of the color $c$, and that edge receives a new color during the interchange. The edge that was previously colored $c$ is horizontal when $e$ is horizontal, and vertical when $e$ is vertical. Indeed, if $e$ is precolored, then the edges adjacent to $e$ in $L$ must receive their colors from the canonical coloring. Hence, recoloring the edge that is newly colored $5$ cannot create any conflict.

        \item If an edge $e \in E_0$ is precolored with $c \neq 5$ and has exactly two neighboring edges colored $c$, then these two neighboring edges are vertical when $e$ is horizontal, and horizontal when $e$ is vertical. Consequently, the endpoints of these three edges form a square. This square is colored with exactly two colors: the two edges in $E_0$ receive color $5$, while the other two edges receive color $c$. We then swap the two colors appearing on this square (or, using the terminology of the proof of Theorem~\ref{general}, perform a rotation). 
    \end{enumerate}

    All these modifications result in an admissible coloring except at the vertices of the square $C_0$. After coloring the grid $L$, we only changed the coloring on internal edges of some brick-neighborhoods.

    Finally, we correct the coloring at the vertices of the square $C_0$. Initially, all edges of $C_0$ are colored with $5$. This is not affected by the coloring of $L$. Among the five possible modifications described above, only the last one can change it ($e \in E_0$ is precolored with $c \neq 5$ and $e$ has exactly two neighboring edges colored $c$). If this is the case, it means that one edge, or possibly two opposite edges, of the edges $(2k,1)(0,1), (1,2l)(1,0), (2k,2l-1)(0,2l-1)$ and $(2k-1,2l)(2k-1,0)$ is/are precolored with the colors $3, 1, 4$ and $2$, respectively.

    We present that the coloring can be corrected in one case of the four possible cases; the other cases can be handled similarly. Suppose that the edge $(2k,1)(0,1)$ is precolored with color $3$. We then perform a rotation on the square with the vertex set $\{(2k,0), (2k,1), (0,1), (0,0)\}$. As a result, the edges $(2k,1)(0,1)$ and $(2k,0)(0,0)$ receive color $3$, while the edges $(2k,0)(2k,1)$ and $(0,0)(0,1)$ receive color $5$. We then recolor $(2k,0)(2k,1)$ with color $1$ and $(0,0)(0,1)$ with color $2$. The only vertices at which a conflict could be created are $(2k,0)$, $(2k,1)$, $(0,1)$, and $(0,0)$. At these vertices the incident colors become, respectively, $\{1,2,3,5\}$, $\{1,2,3,4\}$, $\{1,2,3,4\}$, and $\{1,2,3,5\}$ on the affected edges and their unchanged neighboring edges, so the coloring becomes proper.

    Now, we examine whether, when an edge of $C_0$ is colored with color $5$, its color can be replaced by another color in such a way that no new conflict is created and the coloring is improved. If all four edges of $C_0$ are colored with $5$, we apply this procedure to two opposite edges. If three edges of $C_0$ are colored with $5$, we apply it to the edge opposite the one that is not colored with $5$. As before, it suffices to consider only one case in detail, since all remaining cases follow analogously by symmetry.

    Suppose that the edge $(2k,0)(0,0)$ is colored with $5$, and we wish to modify its color. We may assume that the edges $(2k,2l)(2k,0)$ and $(0,2l)(0,0)$ are also colored with $5$, since, by the previous paragraph, this condition is satisfied whenever we attempt to replace the color of $(2k,0)(0,0)$. Furthermore, any edge having at least one endpoint in the set $\{(0,0), (2k,0)\}$ was certainly not precolored, since $C_0$ is not contained in the brick-neighborhood of any precolored edge.

    At each of the vertices $(2k,0)$ and $(0,0)$, there are two incident edges colored with $5$ and two incident edges whose colors are determined during the coloring of the grid $L$. In the canonical coloring, the edges incident with $(2k,0)$ have colors $2$ and $3$, while the edges incident with $(0,0)$ have colors $1$ and $3$. If this has not changed, then the edge $(2k,0)(0,0)$ may be assigned color $4$.

    The above color assignments could have changed only in the following cases: if the edge $(0,1)(1,1)$ is precolored with $4$, if $(1,0)(1,1)$ is precolored with $2$, if $(2k-1,1)(2k,1)$ is precolored with $1$, or if $(2k-1,0)(2k-1,1)$ is precolored with $4$. In all other cases, either the colors of the edges remain unchanged or only the colors of the edges incident with one of the vertices are interchanged, by the algorithm described in the proof of Theorem~\ref{general}. By the distance constraint, the edges $(0,1)(1,1)$ and $(2k-1,1)(2k,1)$ cannot be precolored simultaneously with another edge, so there are 6 possible cases. The table below specifies the new color of the edge $(2k,0)(0,0)$ for every possible case.

\begin{table}[ht]
\centering
\begin{tabular}{|c|c|c|c|c|}
\hline
\makecell{Precoloring \\ near $(2k,0)$} &
\makecell{Precoloring \\ near $(0,0)$} &
\makecell{Colors \\ at $(2k,0)$} &
\makecell{Colors \\ at $(0,0)$} &
\makecell{New color of\\ $(2k,0)(0,0)$} \\
\hline
None & None & 2,3 & 1,3 & 4 \\
\hline
None & $(0,1)(1,1)$ & 2,3 & 1,2 & 4 \\
\hline
None & $(1,0)(1,1)$ & 2,3 & 3,4 & 1 \\
\hline
$(2k-1,0)(2k-1,1)$ & None & 3,4 & 1,3 & 2 \\
\hline
$(2k-1,0)(2k-1,1)$ & $(1,0)(1,1)$ & 3,4 & 3,4 & 1 \\
\hline
$(2k-1,1)(2k,1)$ & None & 1,2 & 1,3 & 4 \\
\hline
\end{tabular}
\caption{Cases determined by the possible precolored edges near $(2k,0)$ and $(0,0)$.}
\label{tab:precoloring-cases}
\end{table}

\end{proof}

\section{Concluding remarks}

Several interesting problems remain open. The most natural question is whether the main theorem holds for two graphs even if $\Delta(G) \neq \Delta(H)$, or in other words, whether the theorem holds without the degree conditions.

\begin{conjecture}
    Let $G$ and $H$ be Class 1 (or bipartite) graphs with maximum degrees $\Delta(G)$ and $\Delta(H)$, respectively. Suppose that the color of some edges of the Cartesian product $G \square H$ is prescribed using up to $\Delta(G) + \Delta(H)$ colors such that the distance between any two precolored edges is at least three. Then this precoloring is extendable.
\end{conjecture}

We studied the extendability of precolorings on the product of even cycles and on the product of odd cycles; however, the case of the product of an even cycle and an odd cycle remains unsolved.

\sloppy
\begin{conjecture}
    Suppose that the color of some edges of the Cartesian product $C_{2k} \square C_{2l+1}$ is prescribed using at most $4$ colors such that the distance between any two precolored edges is at least three. Then this precoloring is extendable.
\end{conjecture}
\fussy

\noindent
Another interesting question is whether Theorem \ref{bipartite-times-edge} holds for any graph $G$.

\begin{conjecture}
    Let $G$ be a graph with maximum degree $\Delta(G)$. Suppose that the color of some edges of the Cartesian product $G \square K_{2}$ is prescribed using at most $\Delta(G)+1$ colors such that the distance between any two precolored edges is at least three. Then, this precoloring is extendable.
\end{conjecture}

\section*{Statements and declarations}
The first author is partially supported by the Counting in Sparse Graphs Lendület Research Group.
The research of the second author was partially supported by the National Research, Development and Innovation Office NKFIH. 
The authors declare that they have no conflict of interest.
Data sharing is not applicable to this article, as no datasets were generated or analyzed during the current study.

\end{document}